\documentclass{amsart}    
\usepackage {amsthm, amsmath, amscd, amssymb}
\theoremstyle{plain}

\newtheorem{theorem}{Theorem}[section]
\newtheorem{conjecture}[theorem]{Conjecture}
\newtheorem{lemma}[theorem]{Lemma}
\newtheorem{proposition}[theorem]{Proposition}
\newtheorem{corollary}[theorem]{Corollary}

\newtheorem{remark}[theorem]{Remark}
\newtheorem*{property}{Property (*)}
\theoremstyle{remark}

\theoremstyle{definition}

\usepackage[mathscr]{eucal}
\usepackage{graphics, graphpap}
\usepackage{array, tabularx, longtable}

\title[The Kontsevich-Soibelman conjecture]{\bf A proof of the $\ell$-adic version of the integral identity conjecture for polynomials }  
\author{L\^e Quy Thuong}
\address{Institut de Math\'ematique de Jussieu, UMR 7586 CNRS, 4 place Jussieu, 75005 Paris, France {\rm(current)}}
\email{leqthuong@math.jussieu.fr}
\address{Department of Mathematics, Vietnam National University, 334 Nguyen Trai Street, Hanoi, Vietnam}
\email{thuonglq@vnu.edu.vn}

\begin{document}           
\begin{abstract}
It is well known that the integral identity conjecture is of prime importance in Kontsevich-Soibelman's theory of motivic Donaldson-Thomas invariants for non-commutative Calabi-Yau threfolds. In this article we consider its numerical version and make it a complete demonstration in the case where the potential is a polynomial and the ground field is algebraically closed. The foundamental tool is the Berkovich spaces whose crucial point is how to use the comparison theorem for nearby cycles as well as the K\"{u}nneth isomorphism for cohomology with compact support.
\end{abstract}
\maketitle                 

\section{Introduction}\label{sec1}
Let us start by outlining due to \cite{KS2} on the concept of motivic Donaldson-Thomas invariants that concern the integral identity conjecture. These invariants is introduced in \cite{KS} in the framework for Calabi-Yau threfolds and the motivic Hall algebra. The latter generates the derived Hall algebra of To\"{e}n \cite{To}. 

Let $\mathcal{C}$ be an ind-constructible triangulated $A_{\infty}$-category over a field $\kappa$. By giving a constructible stability condition on $\mathcal{C}$ one considers a collection of full subcatgories $\mathcal{C}_V\subset\mathcal{C}$, with $V$ strict sectors in $\mathbb R^2$. The stability condition depends on homomorphisms $cl: K_0(\mathcal{C})\to\Gamma$ and $Z:\Gamma\to\mathcal{C}$, where $\Gamma$ is a free abelian group endowed with a skew-symmetric integer-valued bilinear form $\langle ,\rangle$. A choice of $V$ gives rise to a cone $C(V,Z)$ contained in $\Gamma\otimes\mathbb{R}$ to which one associates a complete motivic Hall algebra $\hat{H}(\mathcal{C}_V)$. Define $A_V^{\text{Hall}}$ invertible in $\hat{H}(\mathcal{C}_V)$ as characteristic functions of the stacks of objects of $\mathcal C_V$. The {\it generic} elements satisfy the Factorization Property
$$A_V^{\text{Hall}}=A_{V_1}^{\text{Hall}}\cdot A_{V_2}^{\text{Hall}}$$
with $V=V_1\sqcup V_2$ and the decomposition taken clockwisely.

If the field $\kappa$ has characteristic zero, motivic quantum torus $\mathcal{R}_{\mathcal{C}}$ is defined to be an associative algebra generated by symbols $\hat{e}_{\gamma}$, for $\gamma$ in $\Gamma$, with the usual relations 
$$\hat{e}_{\gamma_1}\hat{e}_{\gamma_2}=[\mathbb{A}_\kappa^1]^{\frac 1 2 \langle\gamma_1,\gamma_2\rangle}\hat{e}_{\gamma_1+\gamma_2},\ \hat{e}_0=1,$$
where $[\mathbb{A}_\kappa^1]^{\frac 1 2}$ is the square root of $[\mathbb{A}_{\kappa}^1]$. The coefficient ring $C_0$ for the quantum torus $\mathcal{R}_{\mathcal{C}}$ can be any commutative ring, where the two most important candidates should be a certain localization of the Grothendieck ring of algebraic $\kappa$-varieties and its $\ell$-adic version.

By choosing in addition the so-called orientation data (its existence depends on another conjecture) and using Denef-Loeser's theory of motivic Milnor fiber (e.g. the motivic Thom-Sebastiani theorem) of the potential of an object of the category $\mathcal{C}$, by  \cite[Sec. 6]{KS}, there is a map $\Phi_V:\hat{H}(\mathcal{C}_V)\to \mathcal{R}_{\mathcal{C}_V}$ for each $V$, which is nice enough in the sense that if it was a homomorphism the Factorization Property would be preserved. This is in fact obstructed because of the lack of an assertion of the integral identity. In the case where the above $C_0$ is a certain localization of the ring $\mathscr M_{\kappa}^{\hat{\mu}}$, one faces to the full version of the integral identity conjecture. If well passed, $A_V^{\text{mot}}:=\Phi_V(A_V^{\text{Hall}})$ would be invariants in the category of non-commutative Calabi-Yau threfolds, namely {\it motivic} Donaldson-Thomas invariants.  Also, if $C_0$ is a variant of the Grothendieck ring $K_0(D^b_{\text{constr},\text{aut}}(\text{Spec}(\kappa),\mathbb Q_{\ell}))$, one meets the $\ell$-adic version of the conjecture, and in this case, the corresponding invariants are {\it numerical} Donaldson-Thomas invariants.
 
In the context of non-archimedean complete discretely valued fields $K$ of equal characteristic zero, with valuation ring $R$ and residue field $\kappa$, Kontsevich-Soibelman define in \cite{KS} the motivic Milnor fiber $\mathcal S_{\mathfrak f, \mathbf x}$ of a formal function $\mathfrak f: \mathfrak X\to \text{Spf}(R)$ at a closed point $\mathbf x$ of the reduction $\mathfrak X_0$. To do this, they use Denef-Loeser's formula on the motivic nearby cycle of a regular function (cf. \cite{DL1, DL2}) as well as the fact that resolution of singularities of $(\mathfrak X, \mathfrak X_0)$ exists (see Temkin \cite{Tem1}). Let $\int_{\mathscr U}$ be the forgetful morphism for $\mathscr U$ a subvariety of $\mathfrak X_0$.

\begin{conjecture}[Integral identity \cite{KS}]\label{conj1}
Let $f$ be in $\kappa[[x,y,z]]$ invariant by the $\kappa^{\times}$-action of weight $(1,-1, 0)$ with $f(0,0,0)=0$. Denote by $\mathfrak X$ the formal neighborhood of $\mathbb A_{\kappa}^{d_1}$ in $\mathbb A_{\kappa}^d$ whose structural morphism $\hat{f}$ is induced by $f(x,y,z)$. Denote by $\mathfrak Z$ the formal neighborhood of $0$ in $\mathbb A_{\kappa}^{d_3}$ whose structural morphism $\hat{f}_{\mathfrak Z}$ is induced by $f(0,0,z)$. Then, the identity $\int_{\mathbf x\in\mathbb A_{\kappa}^{d_1}}\mathcal{S}_{\hat{f},\mathbf x}=[\mathbb A_{\kappa}^1]^{d_1}\mathcal{S}_{\hat{f}_{\mathfrak Z},0}$ holds in $\mathscr M_{\kappa}^{\hat{\mu}}$.
\end{conjecture}

Notice that we proved  in \cite{Thuong} the {\it regular} version for a composition with a polynomial in two variables and for a function of Steenbrink type. The purpose of the present article is to show that the $\ell$-adic version of the integral identity conjecture holds if the series $f$ is a polynomial and the ground field $\kappa$ is an algebraically closed field of characteristic zero. Let $R\psi$ denote the nearby cycles functor. This functor was defined earlier in \cite{Ber2,Ber} and it will be recalled here in Subsection \ref{ncf}. 

\begin{theorem}\label{maintheorem}
Let $\kappa$ be an algebraically closed field. If $f$ is in $\kappa[x,y,z]$ invariant by the $\kappa^{\times}$-action of weight $(1,-1, 0)$ with $f(0,0,0)=0$, there is a canonical quasi-isomorphism of complexes: $R\Gamma_c(\mathbb A_{\kappa}^{d_1},R\psi_{\hat{f}}\mathbb Q_{\ell}|_{\mathbb A_{\kappa}^{d_1}})\stackrel{qis}{\to} R\Gamma_c(\mathbb A_{\kappa}^{d_1},\mathbb Q_{\ell}){\otimes} (R\psi_{\hat{f}_{\mathfrak Z}}\mathbb Q_{\ell})_0$.
\end{theorem}

As an approach, we follow Kontsevich-Soibelman's idea in \cite[Prop. 9]{KS} using Berkovich spaces. The fundamental tools are the comparison theorem for nearby cycles and the K\"{u}nneth isomorphism for \'etale cohomology with compact support. 

The result in this article is part of the author's thesis. He thanks his advisor Fran\c cois Loeser for such an interesting subject as well as many valuable suggestions and much patience. He thanks Vladimir Berkovich and Antoine Ducros for their answers to questions on Berkovich spaces. Especially, Ducros read carefully the earlier drafts of the manuscript and pointed out a serious mistake, so that the author can introduce this complete version.

\section{Preliminaries on the Berkovich spaces}\label{sec2}

\subsection{Notation}
Let $K$ be a non-archimedean complete discretely valued field $K$ of equal characteristics zero, with valuation ring $R$, maximal ideal $\mathfrak m$ and residue field $\kappa=R/\mathfrak m$.

Let $\mathbb A_{K,\text{Ber}}^n$ be the $n$-dimensional $K$-analytic affine space, which is by definition the set $\mathcal M(K[T_1,\dots,T_n])$ of all multiplicative seminorms on the ring of polynomials $K[T_1,\dots,T_n]$ whose restriction to $K$ is bounded (see \cite{Ber0}). We define a norm on $K$ by $|\xi|:=c^{\text{val}(\xi)}$ with $c\in (0,1)$ fixed, and a norm on $\mathbb A_{K,\text{Ber}}^n$ by $|x|:=\max_{1\leq i\leq n}|x_i|$ for $x=(x_1,\dots,x_n)$. The subspace of $\mathbb A_{K,\text{Ber}}^n$ defined by $|x|\leq 1$ is called the $n$-dimensional unit closed disc and denoted by $E^n(0;1)$, while the corresponding open one is written as $D^n(0;1)$


\subsection{From special formal schemes to analytic spaces}\label{sec2.1}
Let us remark that the main result of this article will only concern formal $R$-schemes topologically of finite type. It is however better to recall some preliminaries on the Berkovich spaces in a larger category that consists of special formal $R$-schemes. 

A topological $R$-algebra $\mathcal A$ is said to be {\it special} if $\mathcal A$ is a Noetherian adic ring such that, if $\mathcal J$ is an ideal of definition of $\mathcal A$, the quotient rings $\mathcal A/\mathcal J^n$, $n\geq 1$, are finitely generated over $R$. By \cite{Ber}, a topological $R$-algebra $\mathcal A$ is special if and only if it is topologically $R$-isomorphic to a quotient of the special $R$-algebra $R\{T_1,\dots, T_n\}[[S_1,\dots,S_m]]$. An adic $R$-algebra $\mathcal A$ is {\it topologically finitely generated over} $R$ if it is topologically $R$-isomorphic to a quotient algebra of the algebra of restricted power series $R\{T_1,\dots,T_n\}$. Evidently, any topologically finitely generated $R$-algebra is a special $R$-algebra.  

A formal $R$-scheme $\mathfrak X$ is said to be {\it special} if $\mathfrak X$ is a separated Noetherian adic formal scheme and if it is a finite union of affine formal schemes of the form $\text{Spf}(\mathcal A)$ with $\mathcal A$ a special $R$-algebras. A formal $R$-scheme $\mathfrak X$ is {\it topologically of finite type} if it is a finite union of affine formal schemes of the form $\text{Spf}(\mathcal A)$ with $\mathcal A$ topologically finitely generated $R$-algebras. It is a fact that the category of separated topologically of finite type formal $R$-schemes is a full subcategory of the category of $R$-special formal schemes, and both admit fiber products. 

A morphism $\varphi: \mathfrak{Y}\to\mathfrak{X}$ of special formal schemes is of {\it locally finite type} if locally it is isomorphic to a morphism of the form $\text{Spf} (\mathcal B)\to \text{Spf} (\mathcal A)$ with $\mathcal B$ topologically finitely generated over $\mathcal A$. The morphism $\varphi$ is of {\it finite type} if it is a quasicompact morphism of locally finite type.

Due to \cite{Ber}, there is a canonical functor $\mathfrak X \mapsto \mathfrak X_{\eta}$ from the category of special formal $R$-schemes to that of (Berkovich) $K$-analytic spaces. In the affine case, the interpretation of this functor is explicit. Namely, if 
$$\mathfrak{X}=\text{Spf} \Big(R\{T_1,\dots,T_n\}[[S_1,\dots,S_m]]\Big),$$
one has
\begin{align*}
\mathfrak X_{\eta}
=E^n(0;1)\times D^m(0;1).
\end{align*}
Also, if $\mathfrak{X}=\text{Spf} (\mathcal A)$, where $\mathcal A$ is a quotient of $R\{T_1,\dots,T_n\}[[S_1,\dots,S_m]]$ by an ideal $\mathcal I$, then 
$\mathfrak{X}_{\eta}$ is the closed $K$-analytic subspace of $X=E^n(0;1)\times D^m(0;1)$ defined by the subsheaf of ideals $\mathcal I\mathscr{O}_{X}$. 

Generally, $\mathfrak X_{\eta}$ is defined by glueing in an appropriate manner of analytic spaces corresponding to affine formal schemes which covers $\mathfrak X$ (see \cite{Ber}).

\begin{remark}\label{rk1}
(i) The functor $\mathfrak{X}\mapsto\mathfrak{X}_{\eta}$ takes a formal scheme topologically of finite type to a paracompact analytic space, and this functor commutes with fiber products.

(ii) The functor $\mathfrak{X}\mapsto\mathfrak{X}_{\eta}$ takes a morphism of finite type $\varphi:\mathfrak{Y}\to\mathfrak{X}$ to a compact morphism of $K$-analytic spaces $\varphi_{\eta}:\mathfrak{Y}_{\eta}\to\mathfrak{X}_{\eta}$. If $\varphi$ is finite (resp. flat finite), so is $\varphi_{\eta}$.
\end{remark}

\subsection{The reduction map}
For a special formal $R$-scheme $\mathfrak X$, we denote by $\mathfrak X_0$ the closed subscheme of $\mathfrak X$ defined by the largest ideal of definition of $\mathfrak X$. Note that $\mathfrak X_0$ is a reduced Noetherian scheme, that the correspondence $\mathfrak X\mapsto\mathfrak X_0$ is functorial, and that the natural closed immersion $\mathfrak X_0\to \mathfrak X$ is a homeomorphism. Moreover, the {\it reduction} $\mathfrak X_0$ is also a separated $\kappa$-scheme of finite type.

We now recall the construction of the reduction map in the affine case, that is for $\mathfrak{X}=\text{Spf}(\mathcal A)$ with $\mathcal A$ being an adic special $R$-algebra. Notice that Berkovich did this work in \cite{Ber2, Ber} for any special formal $R$-scheme. The construction of the reduction map $\pi: \mathfrak{X}_{\eta}\to\mathfrak{X}_0$ for $\mathfrak{X}=\text{Spf}(\mathcal A)$ runs as follows. Remark that each point $x$ of $\mathfrak{X}_{\eta}$ 
defines a continuous character $\chi_x: \mathcal A\to \mathscr{H}(x)$. In its turn, $\chi_x$ defines a character $\widetilde{\chi}_x:\mathcal A_0=\mathcal A/\mathcal J \to\widetilde{\mathscr{H}(x)}$, where $\mathcal J$ is the largest ideal of definition of $\mathcal A$. Then we assign $\pi(x)$ to the kernel of $\widetilde{\chi}_x$, which is a prime ideal of $\mathcal A_0$. This definition guarantees the compatibility of the reduction map with open immersion in the following meaning. If $\mathfrak Y$ is an open formal scheme of $\mathfrak X$, then the reduction maps for $\mathfrak X$ and $\mathfrak Y$ are compatible and $\mathfrak Y_{\eta}\cong \pi^{-1}(\mathfrak Y_0)$.

\subsection{\'Etale cohomology of analytic spaces}
The theory of \'etale cohomology for Berkovich spaces (also called non-archimedean analytic spaces) is sharply developed in the long article \cite{Ber1}. Note that the groups $H^*(Y,\mathbb Z_{\ell})$ and $H^*(Y,\mathbb Q_{\ell})$ in the sense of derived functors are irrelevant, i.e. roughly speaking, they do not satisfy some ``nice" properties which a cohomology theory should have. Grothendieck however pointed out that the following groups are relevant
$$\projlim H^*(Y,\mathbb Z/\ell^n\mathbb Z)\quad \text{and}\quad (\projlim H^*(Y,\mathbb Z/\ell^n\mathbb Z))\otimes_{\mathbb Z_{\ell}} \mathbb Q_{\ell}.$$
Thus from now on, we shall only consider these groups and denote them by $H^*(Y,\mathbb Z_{\ell})$ and $H^*(Y,\mathbb Q_{\ell})$, respectively (cf. \cite{Ducros}, \cite{NS}). The same also holds for cohomology with compact support (cf. \cite{Ducros}, \cite{HL}). Namely, 
\begin{align*}
H_c^*(Y,\mathbb Z_{\ell})&:= (\projlim H_c^*(Y,\mathbb Z/\ell^n\mathbb Z)),\\
H_c^*(Y,\mathbb Q_{\ell})&:= (\projlim H_c^*(Y,\mathbb Z/\ell^n\mathbb Z))\otimes_{\mathbb Z_{\ell}} \mathbb Q_{\ell}.
\end{align*}

Let $\widehat{K^s}$ be the completion of a separable closure of $K$. For a $K$-analytic space $X$, there is a canonical morphism $b: \overline{X}:=X\widehat{\otimes}_K\widehat{K^s}\to X$. Now fix such an $X$ and consider all the subspaces of its. If $Y$ is an analytic subspace of the $X$, denote by $\overline{Y}$ or by $Y\widehat{\otimes}_K\widehat{K^s}$ the preimage of $Y$ in $\overline{X}$ under $b$. The following are two of properties of the functor $Y\mapsto H_c^*(\overline{Y},\mathbb Q_{\ell})$ according to \cite[Prop. 5.2.6, Cor. 7.7.3]{Ber1}.

\begin{proposition}[Berkovich \cite{Ber1}]\label{prop2.2}
Let $Y$, $Y'$ be locally closed analytic subspaces of a given $K$-analytic space $X$.
\begin{itemize}
\item[(i)] If $U$ is an open subspace of $Y$, $V:=Y\setminus U$, there is an exact sequence 
\begin{align*}
\cdots\to H_c^m(\overline{V},\mathbb Q_{\ell})\to H_c^{m+1}(\overline{U},\mathbb Q_{\ell})\to H_c^{m+1}(\overline{Y},\mathbb Q_{\ell})\to H_c^{m+1}(\overline{V},\mathbb Q_{\ell})\to \cdots.
\end{align*}
\item[(ii)] There is a canonical K\"{u}nneth isomorphism of complexes
$$R\Gamma_c(\overline{Y},\mathbb Q_{\ell})\otimes R\Gamma_c(\overline{Y'},\mathbb Q_{\ell})\cong R\Gamma_c(\overline{Y\times Y'},\mathbb Q_{\ell}).$$
\end{itemize}
\end{proposition}


\subsection{The nearby cycles functor}\label{ncf}
A morphism $\varphi: \mathfrak{Y}\to\mathfrak{X}$ of special formal $R$-schemes is called {\it \'etale} if for any ideal of definition $\mathcal{J}$ of $\mathfrak{X}$ the morphism of schemes $(\mathfrak{Y},\mathscr{O}_{\mathfrak{Y}}/\mathcal{J}\mathscr{O}_{\mathfrak{Y}})\to (\mathfrak{X},\mathscr{O}_{\mathfrak{X}}/\mathcal{J})$ is \'etale. The reduction $\mathfrak X_0$ being the closed subscheme of $\mathfrak X$ defined by the largest ideal of definition of $\mathfrak X$, thus if the morphism $\varphi: \mathfrak{Y}\to\mathfrak{X}$ is \'etale, the induced morphism $\varphi_0: \mathfrak{Y}_0\to\mathfrak{X}_0$ is \'etale. 

By \cite{Ber1}, a morphism of $K$-analytic spaces $\varphi : Y \to X$ is {\it \'etale} if for each point $y\in Y$ there exist open neighborhoods $V$ of $y$ and $U$ of $\varphi(y)$ such that $\varphi$ induces a finite \'etale morphism $\varphi: V \to U$. By a finite \'etale morphism $\varphi: V \to U$ one means that for each affinoid domain $W = \mathcal M(\mathcal A)$ in $U$, the preimage $\varphi^{-1}(W) =\mathcal M(\mathcal B)$ is an affinoid domain and $\mathcal B$ is a finite \'etale $\mathcal A$-algebra. A morphism of $K$-analytic spaces $\varphi: Y\to X$ is called {\it quasi-\'etale} if for any point $y\in Y$ there exist affinoid domains $V_1,\dots,V_n\subset Y$ such that $V_1\cup\dots\cup V_n$ is a neighborhood of $y$ and each $V_i$ may be identified with an affinoid domain in a $K$-affinoid space \'etale over $X$. By definition, \'etale morphisms are also quasi-\'etale. 

\begin{lemma}[Berkovich \cite{Ber}, Prop. 2.1]\label{lem2.2}
Assume that $\varphi: \mathfrak{Y}\to \mathfrak{X}$ is an \'etale morphism of special formal $R$-schemes. Then the following hold:
\begin{itemize}
  \item[(i)] $\varphi_{\eta}(\mathfrak{Y}_{\eta})=\pi^{-1}(\varphi_0(\mathfrak{Y}_0))$, consequently $\varphi_{\eta}(\mathfrak{Y}_{\eta})$ is a closed analytic domain in $\mathfrak{X}_{\eta}$.
  \item[(ii)] The induced morphism $\frak{Y}_{\eta}\to\frak{X}_{\eta}$ of $K$-analytic spaces is quasi-\'etale.
\end{itemize}
\end{lemma}

For a $K$-analytic space $X$, let $X_{\text{q\'et}}$ denote the quasi-\'etale site of $X$ as in \cite{Ber2}. The quasi-\'etale topology on $X$ is the Grothendieck topology on the category of quasi-\'etale morphisms $U\to X$ generated by the pretopology for which the set of coverings of $(U\to X)$ is formed by the families $\{f_i: U_i\to U\}_{i\in I}$ such that each point of $U$ has a neighborhood of the form $f_{i_1}(V_1)\cup\dots\cup f_{i_n}(V_n)$ for some affinoid domains $V_1\subset U_{i_1},\dots,V_n\subset U_{i_n}$. There is a morphism of sites $\mu: X_{\text{q\'et}}\to X_{\text{\'et}}$. Denote by $X_{\text{q\'et}}^{\sim}$ the category of sheaves of sets on $X_{\text{q\'et}}$. The functor $\mu^*:X_{\text{\'et}}^{\sim}\to X_{\text{q\'et}}^{\sim}$ is a fully faithful functor (cf. \cite{Ber2}).

Let $\mathfrak{X}$ be a special formal $R$-scheme. By \cite{Ber2}, the correspondence $\mathfrak{Y}\mapsto \mathfrak{Y}_0$ induces an equivalence between the category of formal schemes \'etale over $\mathfrak{X}$ and the category of schemes \'etale over $\mathfrak{X}_0$. We fix the functor $\mathfrak{Y}_0\mapsto\mathfrak{Y}$ which is inverse to the previous correspondence $\mathfrak{Y}\mapsto\mathfrak{Y}_0$. The composition of the functor $\mathfrak{Y}_0\mapsto\mathfrak{Y}$ with the functor $\mathfrak{Y}\mapsto \mathfrak{Y}_{\eta}$ induces a morphism of sites $\nu: \mathfrak{X}_{\eta\text{q\'et}}\to\mathfrak{X}_{0\text{\'et}}$. By \cite{Ber}, this construction also holds over a separable closure $K^s$ of $K$, therefore we shall also denote by $\nu$ the corresponding morphism of sites $\mathfrak{X}_{\overline{\eta}\text{q\'et}}\to\mathfrak{X}_{\overline{0}\text{\'et}}$, where $\mathfrak{X}_{\overline{\eta}}:=\mathfrak{X}_{\eta}\widehat{\otimes}_K\widehat{K^s}$ and $\mathfrak{X}_{\overline{0}}:=\mathfrak{X}_0\otimes_{\kappa} \kappa^s$.

Now consider the composition of the functors $\mu^*: \mathfrak{X}_{\overline{\eta}\text{\'et}}^{\sim}\to \mathfrak{X}_{\overline{\eta}\text{q\'et}}^{\sim}$ and $\nu_*: \mathfrak{X}_{\overline{\eta}\text{q\'et}}^{\sim}\to \mathfrak{X}_{\overline{0}\text{\'et}}^{\sim}$, namely $\nu_*\mu^*: \mathfrak{X}_{\overline{\eta}\text{\'et}}^{\sim}\to \mathfrak{X}_{\overline{0}\text{\'et}}^{\sim}$. This resulting functor composing with the pullback (or inverse image) functor of the canonical morphism $\mathfrak X_{\overline{\eta}}\to \mathfrak X_{\eta}$ yields a functor $\psi: \mathfrak{X}_{\eta\text{\'et}}^{\sim}\to \mathfrak{X}_{\overline{0}\text{\'et}}^{\sim}$, which is called the {\it nearby cycles functor} (see \cite{Ber2, Ber}). It is a left exact functor, thus we can involve right derived functors $R^i\psi: \mathbf{S}(\mathfrak{X}_{\eta})\to \mathbf{S}(\mathfrak{X}_{\overline{0}})$ and $R\psi: D^+(\mathfrak{X}_{\eta})\to D^+(\mathfrak{X}_{\overline{0}})$, the latter is exact while the others are right exact functors. If necessary, we can write $R^i\psi_{\mathfrak f}$ and $R\psi_{\mathfrak f}$ labeling $\mathfrak f$ the structural morphism of $\mathfrak{X}$.

\begin{lemma}[Berkovich \cite{Ber}, Cor. 2.3]\label{lem2.3}
Let $\varphi: \mathfrak{Y}\to \mathfrak{X}$ be an \'etale morphism of special formal $R$-schemes and $F$ in $\mathbf{S}(\mathfrak{X}_{\eta})$. Then for any $m\geq 0$ we have $(R^m\psi F)|_{\mathfrak{Y}_{\overline{0}}}\cong R^m\psi(F|_{\mathfrak{Y}_{\eta}})$. 
\end{lemma}

\subsection{The comparison theorem for nearby cycles}\label{BCTsection}
By \cite[Thm 3.1]{Ber}, the comparison theorem for nearby cycles functor working on a henselian ring $R$. Let $\mathscr E$ be a scheme locally of finite type over $R$ with the structural morphism $f$; and let $\mathscr E_0$ be the zero locus of $f$, which is a $\kappa$-scheme. Then $\mathscr E_0=\widehat{\mathscr E}_0$, where the scheme on the right is the reduction of the completion $\widehat{\mathscr E}$ of the scheme $\mathscr E$. For a subscheme $\mathscr{Y}\subset \mathscr E_0$, let $\widehat{\mathscr E }_{/\mathscr{Y}}$ denote the formal $\mathfrak m$-adic completion of $\widehat{\mathscr E }$ along $\mathscr{Y}$. A result of \cite{Ber} shows that there is a canonical isomorphism of $K$-analytic spaces $(\widehat{\mathscr E }_{/\mathscr{Y}})_{\eta}\cong \pi^{-1}(\mathscr{Y})$, where $\pi$ is the reduction map $\widehat{\mathscr E }_{\eta}\to\mathscr E_0$. For a sheaf $\mathcal{F}\in\mathscr E _{\eta\text{\'et}}^{\sim}$, with $\mathscr E _{\eta}:=\mathscr E \otimes_RK$, let $\widehat{\mathcal{F}}_{/\mathscr{Y}}$ denote the pullback of $\mathcal{F}$ on $(\widehat{\mathscr E }_{/\mathscr{Y}})_{\eta}$. The nearby cycles functor for $\mathscr E$, for $\widehat{\mathscr E}$ and for $(\widehat{\mathscr E}_{/\mathscr Y})_{\eta}$ will be denoted by the same symbol $\psi$. If $\mathscr Y$ is an (ordinary) $\kappa$-scheme, we define $\overline{\mathscr{Y}}:=\mathscr{Y}\otimes_{\kappa} \kappa^s$.

%
\begin{theorem}[Berkovich \cite{Ber}, Thm. 3.1]\label{mainber}
Let $\mathcal{F}$ be an \'etale abelian constructible sheaf on $\mathscr E _{\eta}$. For $i\geq 0$, there is a canonical isomorphism $(R^i\psi\mathcal{F})|_{\overline{\mathscr{Y}}}\cong R^i\psi(\widehat{\mathcal{F}}_{/\mathscr{Y}})$.
\end{theorem}

The previous theorem is widely known as the Berkovich's comparison theorem for nearby cycles, while the full version is in fact stated for both nearby cycles functor and vanishing cycles functor and it is motivated by a conjecture of Deligne. Part of the conjecture claims that the restrictions of the vanishing cycles sheaves of a scheme $\mathscr E$ of finite type over a henselian discrete valuation ring to the subscheme $\mathscr{Y}\subset\widehat{\mathscr E}_0$ depend only on the formal $\mathfrak m$-adic completion $\widehat{\mathscr E }_{/\mathscr{Y}}$ of $\mathscr E$ along $\mathscr{Y}$, and that the automorphism group of $\widehat{\mathscr E}_{/\mathscr{Y}}$ acts on them. By proving this comparison theorem, Berkovich \cite{Ber} provided the positive answer to Deligne's conjecture.

The following corollary runs over any complete discretely valued field.

\begin{corollary}[Berkovich \cite{Ber}, Cor. 3.6]\label{cormain}
Let $\mathscr{S}$ be an $R$-scheme of locally finite type, $\mathfrak{X}$ a special formal $\widehat{\mathscr{S}}$-scheme which is locally isomorphic to the formal $\mathfrak m$-adic completion of a $\mathscr{S}$-scheme of finite type along a subscheme of its reduction, $F$ an \'etale sheaf on $\mathfrak{X}_{\eta}$ locally in the \'etale topology of $\mathfrak{X}$ isomorphic to the pullback of a constructible sheaf on $\widehat{\mathscr{S}}_{\eta}$. Then $R\psi(F)$ is constructible and, for any subscheme $\mathscr{Y}\subset\mathfrak{X}_0$, there is a canonical isomorphism of complexes
$$R\Gamma(\overline{\mathscr{Y}},(R\psi F)|_{\overline{\mathscr{Y}}})\stackrel{\sim}{\to} R\Gamma(\overline{\pi^{-1}(\mathscr{Y})},F).$$
If, in addition, the closure of $\mathscr{Y}$ in $\mathfrak{X}_0$ is proper, there is a canonical isomorphism
$$R\Gamma_c(\overline{\mathscr{Y}},(R\psi F)|_{\overline{\mathscr{Y}}})\stackrel{\sim}{\to} R\Gamma_{\overline{\pi^{-1}(\mathscr{Y})}}(\mathfrak{X}_{\overline{\eta}},F).$$
\end{corollary}


\section{The polynomial $f$ and comparisons}\label{sec3}
From this section, the condition that $\kappa$ is an algebraically closed field will be used because of applying Berkovich's comparison theorem for nearby cycles. Also, $R$ and $K$ will stand for $\kappa[[t]]$ and $\kappa((t))$, respectively.
\subsection{Resetting the data}\label{sec3.1}
Let $f(x,y,z)$ be in $\kappa[x,y,z]$ such that $f(0,0,0)=0$ and $f(\tau x,\tau^{-1}y,z)=f(x,y,z)$ for $\tau\in \kappa^{\times}$. Let us consider the following $R$-schemes with the structural morphisms
\begin{equation}\label{eiffel}
\begin{aligned}
\mathscr E:=\text{Spec}(R[x,y,z]/(f(x,y,z)-t))&\to \text{Spec}(R),\\
\mathscr W:=\text{Spec}(R[z]/(f(0,0,z)-t))&\to \text{Spec}(R)
\end{aligned}
\end{equation}
given by $t=f(x,y,z)$, $t=f(0,0,z)$, respectively. Note that $\mathbb A_{\kappa}^{d_1}$ is a closed subvariety of $\kappa$-variety $\mathscr E_0=f^{-1}(0)$. We have identities $\mathfrak X= \widehat{\mathscr E }_{/\mathbb A_{\kappa}^{d_1}}$ and $\mathfrak Z=\widehat{\mathscr W}_{/0}$, where the formal schemes on the left hand sides were already defined in first section. 

Consider the reduction maps 
$\pi:\mathfrak X_{\eta}\to \mathfrak X_0$ and 
$\pi_{\mathscr W}: \mathfrak Z_{\eta}\to \mathfrak Z_0$.

\subsection{Applying the comparison theorem}\label{sec3.2}
Let $\mathbf{f}$ be the homogenization of $f$, i.e. $\mathbf{f}(x,y,z,\xi)$ is homogeneous in $d+1$ variables with $\mathbf{f}(x,y,z,1)=f(x,y,z)$ and $\deg(\mathbf{f})=\deg(f)=n$. Note that the $R$-scheme
$$\mathbf E:=\text{Proj}\Big(R[x,y,z,\xi]/(\mathbf f(x,y,z,\xi)-t\xi^n)\Big)$$
is locally of finite type. Let us consider the $t$-adic completion $\widehat{\mathbf E}$, which is a formal $R$-scheme canonically glued from the following affine formal $R$-schemes 
\begin{equation}\label{eq100}
\begin{aligned}
&\text{Spf} \Big(R\{\frac x {x_i}, \frac y {x_i},\frac z {x_i},\frac {\xi} {x_i}\}/\big(\mathbf f(\frac x {x_i}, \frac y {x_i},\frac z {x_i},\frac {\xi} {x_i})-t(\frac {\xi} {x_i})^n\big)\Big)&\ i=1,\dots, d_1,\\
&\text{Spf} \Big(R\{\frac x {y_j}, \frac y {y_j},\frac z {y_j},\frac {\xi} {y_j}\}/\big(\mathbf f(\frac x {y_j}, \frac y {y_j},\frac z {y_j},\frac {\xi} {y_j})-t(\frac {\xi} {y_j})^n\big)\Big)&\ j=1,\dots, d_2,\\
&\text{Spf} \Big(R\{\frac x {z_l}, \frac y {z_l},\frac z {z_l},\frac {\xi} {z_l}\}/\big(\mathbf f(\frac x {z_l}, \frac y {z_l},\frac z {z_l},\frac {\xi} {z_l})-t(\frac {\xi} {z_l})^n\big)\Big)&\ l=1,\dots,d_3,\\
&\text{Spf} \Big(R\{\frac x {\xi},\frac y {\xi},\frac z {\xi}\}/\big(f(\frac x {\xi},\frac y {\xi},\frac z {\xi})-t\big)\Big)\cong \widehat{\mathscr E}.
\end{aligned}
\end{equation}
The reduction $\widehat{\mathbf E}_0=\mathbf E_0$ is the hypersurface $\{\mathbf f=0\}$ in the projective space $\mathbb{P}_{\kappa}^d$, it admits the inclusions $\mathbb A_{\kappa}^{d_1}\subset \mathscr E_0\subset \mathbf E_0$.

Let $\tilde{\mathbb A}_{\kappa}^{d_1}$ be the closure of $\mathbb A_{\kappa}^{d_1}$ in $\mathbf E_0$. By construction, the embedding of $\widehat{\mathscr E}$ in $\widehat{\mathbf E}$ is an open immersion of formal $R$-schemes (thus it is an \'etale morphism). By \cite[Cor. 10.9.9]{GD}, the formal $R$-scheme $\mathfrak{X}=\widehat{\mathscr E}_{/\mathbb A_{\kappa}^{d_1}}$ can be identified to the fiber product of $\widehat{\mathscr E}\to \widehat{\mathbf E}$ and $\mathbf X:=\widehat{\mathbf E}_{/\tilde{\mathbb A}_{\kappa}^{d_1}}\to \widehat{\mathbf E}$. Since \'etale morphisms are preserved under base change, the induced morphism $\mathfrak X \to \mathbf X$ is also \'etale (it is even an open immersion). Denote by ${\widehat{\mathbf f}}$ the structural morphism of $\mathbf X$, which is induced by $\mathbf f$. We shall use the following notation
\begin{itemize}
\item[$\star$] $i:\mathfrak{X}_{\overline{\eta}}\to\mathbf X_{\overline{\eta}}$ is the embedding of analytic spaces, 
\item[$\star$] $j:\mathfrak{X}_0\to\mathbf X_0$, $k:\mathbf X_0 \setminus \mathfrak{X}_0\to \mathbf X_0$, $u:\mathbb A_{\kappa}^{d_1}\to \mathfrak{X}_0$ and $v:\mathbb A_{\kappa}^{d_1}\to \mathbf{X}_0$ are the embeddings of $\kappa$-schemes (note that $v=j\circ u$).
\end{itemize}

Let $F$ denote the constant sheaf $(\mathbb{Z}/\ell^n\mathbb{Z})_{\mathfrak{X}_{\overline{\eta}}}$ in $\mathbf S(\mathfrak{X}_{\overline{\eta}})$, $n\geq 1$. By Lemma \ref{lem2.3}, for any $m\geq 0$, we have 
$j^*R^m\psi_{\widehat{\mathbf f}}(i_!F)\cong R^m\psi_{\widehat{ f}}F$, hence
$j_!j^*R^m\psi_{\widehat{\mathbf f}}(i_!F)\cong j_!R^m\psi_{\widehat{ f}}F$. In the latter isomorphism, the complex on the right hand side can be fitted in the exact triangle
$$\to j_!R^m\psi_{\widehat{ f}}F\to R^m\psi_{\widehat{\mathbf f}}(i_!F)\to k_*k^*R^m\psi_{\widehat{\mathbf f}}(i_!F)\to.$$
The functor $v^*$ being exact, we have the following exact triangle 
\begin{align}\label{eq2}
\to u^*R^m\psi_{\widehat{ f}}F\to v^*R^m\psi_{\widehat{\mathbf f}}(i_!F)\to v^*k_*k^*R^m\psi_{\widehat{\mathbf f}}(i_!F)\to.
\end{align}
Observe that the support of the functor $v^*$ is $\mathbb A_{\kappa}^{d_1}$, which is a subset of $\mathfrak X_0$, while that of $k_*k^*$ is $\mathbf X_0 \setminus \mathfrak{X}_0$, and the two subsets $\mathbb A_{\kappa}^{d_1}$ and $\mathbf X_0 \setminus \mathfrak{X}_0$ are disjoint in $\mathbf X_0$. This means $v^*k_*k^*R^m\psi_{\widehat{\mathbf f}}(i_!F)\cong 0$, and one deduces that $R^m\psi_{\widehat{ f}}F|_{\mathbb A_{\kappa}^{d_1}}\cong R^m\psi_{\widehat{\mathbf f}}(i_!F)|_{\mathbb A_{\kappa}^{d_1}}$. The latter leads us to a quasi-isomorphism of complexes
\begin{align}\label{eq3}
R\Gamma_c(\mathbb A_{\kappa}^{d_1},R\psi_{\widehat{ f}}F|_{\mathbb A_{\kappa}^{d_1}})\stackrel{\text{qis}}{\to} R\Gamma_c(\mathbb A_{\kappa}^{d_1},R\psi_{\widehat{\mathbf f}}(i_!F)|_{\mathbb A_{\kappa}^{d_1}}).
\end{align}

Now apply Corollary \ref{cormain} to the nearby cycles functor $R\psi_{\widehat{\mathbf f}}$. For such an $\mathbf f$, the assumptions of that corollary are satisfied: the scheme $\mathbf E $ is of finite type over $R$ and the closure of $\mathbb A_{\kappa}^{d_1}$ in $\mathbf X_0$ is proper as $\mathbf X_0$ is. Let $\tilde{\pi}$ denote the reduction map $\mathbf X_{\eta}\to \mathbf X_0$. One then deduces from Corollary \ref{cormain} that
\begin{align}\label{eq4}
R\Gamma_c(\mathbb A_{\kappa}^{d_1},R\psi_{\widehat{\mathbf f}}(i_!F)|_{\mathbb A_{\kappa}^{d_1}})\stackrel{\sim}{\to} R\Gamma_{\overline{\tilde{\pi}^{-1}(\mathbb A_{\kappa}^{d_1})}}(\mathbf X_{\overline{\eta}},i_!F).
\end{align}

\subsection{Shrinking analytic domains}\label{sec3.3}
Let us consider $R\Gamma_{\overline{\tilde{\pi}^{-1}(\mathbb A_{\kappa}^{d_1})}}(\mathbf X_{\overline{\eta}},i_!F)$ as in (\ref{eq4}). We remark that the analytic space $\mathbf{X}_{\overline{\eta}}$ is the glueing of $A:=\mathfrak X_{\overline{\eta}}$ together with other analytic spaces which correspond to the formal schemes in (\ref{eq100}), each of which is a closed analytic domain in $\mathbf{X}_{\overline{\eta}}$ (Lemma \ref{lem2.2}). Similarly, $\overline{\tilde{\pi}^{-1}(\mathbb A_{\kappa}^{d_1})}$ is the glueing of $X:=\overline{\pi^{-1}(\mathbb A_{\kappa}^{d_1})}$ together with others in the same way. Define $P:=\mathbf{X}_{\overline{\eta}}\setminus A$ and $T:=\mathbf{X}_{\overline{\eta}} \setminus \overline{\tilde{\pi}^{-1}(\mathbb A_{\kappa}^{d_1})}$.

\begin{lemma}\label{lem3.1}
We have a quasi-isomorphism of complexes as follows
\begin{align}\label{eq20}
R\Gamma_{\overline{\tilde{\pi}^{-1}(\mathbb A_{\kappa}^{d_1})}}(\mathbf X_{\overline{\eta}},i_!F)\stackrel{\text{qis}}{\to} R\Gamma_{X}(A,F).
\end{align}
\end{lemma}

\begin{proof}
Let $i_{\alpha}$ be the embedding of an $\widehat{K^s}$-analytic space $\alpha$ in $\mathbf{X}_{\overline{\eta}}$, $i_{\alpha,\beta}$ the embedding of $\alpha$ in $\beta$ (thus $i_{A}=i$), and $B:=A\setminus X$. Now both sides of (\ref{eq20}) can be rewritten as follows
\begin{align*}
R\Gamma_{\overline{\tilde{\pi}^{-1}(\mathbb A_{\kappa}^{d_1})}}(\mathbf X_{\overline{\eta}},i_!F)&\stackrel{\text{qis}}{\to} R\widehat{\mathbf f}_{\overline{\eta} *}\text{Cone} (i_! F\to i_{T*}i_T^*i_! F),\\
R\Gamma_{X}(A,F)&\stackrel{\text{qis}}{\to} R\widehat{f}_{\overline{\eta} *}\text{Cone} (F\to i_{B,A *}i_{B,A}^*F).
\end{align*}

Note that the embeddings $i_P: P\hookrightarrow \mathbf{X}_{\overline{\eta}}$ and $i: A \hookrightarrow \mathbf{X}_{\overline{\eta}}$ altogether give rise to an exact triangle of complexes on $\mathbf{X}_{\overline{\eta}}$:  
\begin{align*}
\to i_{P!}i_P^*\text{Cone}(i_! F\to i_{T*}i_T^*i_! F)\to \text{Cone}(i_! F\to i_{T*}i_T^*i_! F)\\
\stackrel{h}{\longrightarrow} i_*i^*\text{Cone}(i_! F\to i_{T*}i_T^*i_! F)\to.
\end{align*}
The supports of $i_P^*$ and $i_!$ are disjoint, hence $h$ is a quasi-isomorphism. Rewrite $h$ in the form $h: \text{Cone} (i_! F\to i_{T*}i_T^*i_! F)\to \text{Cone} (i_* F\to i_{B*}i_{B,A}^* F)$. The identity $i_B=i\circ i_{B,A}$ implies the following isomorphisms of complexes
\begin{align*}
\text{Cone} (i_* F\to i_{B*}i_{B,A}^* F)&\cong\text{Cone} (i_* F\to i_*i_{B,A *}i_{B,A}^* F)\\
&\cong i_*\text{Cone} (F\to i_{B,A *}i_{B,A}^* F).
\end{align*}
We claim that $R\widehat{\mathbf f}_{\overline{\eta} *}i_*=R\widehat{f}_{\overline{\eta} *}$. Indeed, one deduces from \cite[Cor. 5.2.4]{Ber1} and $\widehat{\mathbf f}_{\overline{\eta}}\circ i=\widehat{f}_{\overline{\eta}}$ that $R\widehat{\mathbf f}_{\eta *}Ri_*=R\widehat{f}_{\eta *}$. That $i_*=i_!$ is as $A$ is closed in $\mathbf X_{\overline{\eta}}$ (cf. Lemma \ref{lem2.2}), while $i_!$ is exact since the stalk $(i_!F)_{\mathbf y}$ is equal to $F_{\mathbf y}$ if $\mathbf y\in A$, and zero otherwise, thus $Ri_*=i_*$. Finally, taking the exact functor $R\widehat{\mathbf f}_{\overline{\eta} *}$ to the quasi-isomorphism $h$ yields a quasi-isomorphism of complexes
$$R\widehat{\mathbf f}_{\overline{\eta} *}\text{Cone} (i_! F\to i_{T*}i_T^*i_! F)\stackrel{\text{qis}}{\to} R\widehat{f}_{\overline{\eta} *}\text{Cone} (F\to i_{B,A*}i_{B,A}^*F),$$
This proves the lemma.
\end{proof}

\subsection{Description of $A$, $X$ and $D$}\label{AXD}
We notice that from now on we shall abuse the notation $x$, $y$, $z$, and others, i.e. we shall use them parallelly with two different senses. Just before $(x,y,z)$ stands for a system of coordinates in $\mathbb A_{\kappa}^d$ ($d=d_1+d_2+d_3$), in what follow it will also denote the corresponding system of coordinates on the analytification $\mathbb A_{K^s}^{d,\text{an}}$. Similarly, if $\tau$ is an element in the group scheme $\mathbb G_{m,\kappa}$, we also write $\tau$ for the corresponding element in $\mathbb G_{m,K^s}^{\text{an}}$.

\begin{lemma}
With $f$ as in Theorem \ref{maintheorem}, the analytic space $A=\mathfrak X_{\overline{\eta}}$ is the inductive limit of the compact domains 
$$A_{\gamma,\epsilon}:=\{(x,y,z)\in\mathbb A_{\widehat{K^s},\text{Ber}}^d: |x|\leq \gamma^{-1}, |y|\leq \gamma\epsilon, |z|\leq \epsilon, f(x,y,z)=t\}$$
with $\gamma,\epsilon$ running over the value group $|(K^s)^*|$ of the absolute value on $K^s$ such that $\gamma,\epsilon \in (0,1)$ and $\gamma,\epsilon\to 1$. In the same way, $X=\overline{\pi^{-1}(\mathbb A_{\kappa}^{d_1})}$ is the inductive limit of 
$$X_{\gamma,\epsilon}:=\{(x,y,z)\in\mathbb A_{\widehat{K^s},\text{Ber}}^d: |x|< \gamma^{-1}, |y|\leq \gamma\epsilon, |z|\leq \epsilon, f(x,y,z)=t\}.$$
\end{lemma}

\begin{proof}
For each $\gamma\in |(K^s)^*|$, choose an element $\tau_{\gamma}$ in $\mathbb G_{m,\kappa}$ such that its corresponding element $\tau_{\gamma}$ in $\mathbb G_{m,\kappa}^{\text{an}}$ takes absolute value $\gamma$. Since $f(\tau_{\gamma} x, \tau_{\gamma}^{-1}y,z)=f(x,y,z)$, the following special $R$-algebras are isomorphic  
$$R\{\tau_{\gamma} x, \tau_{\gamma}^{-1} y,z\}/(f(x,y,z)-t)\cong R\{x,y,z\}/(f(x,y,z)-t).$$
Setting 
$$A_{\gamma}:=\left(\Big(\text{Spf} \frac{R\{\tau_{\gamma} x, \tau_{\gamma}^{-1} y,z\}}{(f(x,y,z)-t)}\Big)_{/\mathbb A_{\kappa}^{d_1}}\right)_{\overline{\eta}},$$
it is clear that 
\begin{align*}
A_{\gamma}&=\{(x,y,z)\in\mathbb A_{\widehat{K^s},\text{Ber}}^d: |\tau_{\gamma} x|\leq 1, |\tau_{\gamma}^{-1} y|< 1, |z|< 1, f(x,y,z)=t\}\\
&=\{(x,y,z)\in\mathbb A_{\widehat{K^s},\text{Ber}}^d: |x|\leq \gamma^{-1}, |y|< \gamma, |z|< 1, f(x,y,z)=t\}
\end{align*}
and that all the spaces $A_{\gamma}$'s, with $\gamma\in |(K^s)^*|$, are analytically isomorphic. The latter implies an analytic isomorphism between any pair $(A_{\gamma}, A_{\gamma'})$ with $\gamma$, $\gamma'$ in $|(K^s)^*|$, and thus one can establish an inductive system 
$$\left\{\{A_{\gamma}\}, \{A_{\gamma}\to A_{\gamma'}\}_{\gamma<\gamma'} : \gamma, \gamma' \in |(K^s)^*|\cap (0,1)\right\}.$$
Then $A$ is exactly the inductive limit of this system $\{A_{\gamma}\}$ when $\gamma\to 1$. On the other hand, the space $\{y: |y|<\gamma\}$ is covered by the compact domains $\{z: |z|\leq \gamma\epsilon\}$ and the space $\{z: |z|<1\}$ is covered by the compact domains $\{z: |z|\leq \epsilon\}$ with $\epsilon\in |(K^s)^*|$ and $0< \epsilon <1$. Therefore $A$ can be viewed as the inductive limit of $A_{\gamma,\epsilon}$'s as above with $\gamma,\epsilon \in |(K^s)^*|\cap (0,1)$ and $\gamma,\epsilon\to 1$. 

The inductive system of $X_{\gamma,\epsilon}$'s whose limit describes $X$ is defined by $X_{\gamma,\epsilon}:=A_{\gamma,\epsilon}\cap X$, transition morphisms induce from those in the system of $A_{\gamma,\epsilon}$'s. 
\end{proof}

We also remark that $D:=\overline{\pi_{\mathscr W}^{-1}(0)}$ is an open and locally compact analytic space, it can be covered by the following compact domains
$$D_{\epsilon}:=\{z\in\mathbb A_{\widehat{K^s},\text{Ber}}^{d_3}: |z|\leq \epsilon, f(0,0,z)=t\},\ \epsilon\in |(K^s)^*|\cap (0,1).$$

\begin{corollary}\label{Ber}
Keeping the assumption of Theorem \ref{maintheorem} and fixing a $\gamma\in |(K^s)^*|\cap (0,1)$, we have

\indent (i) $R\Gamma_c(\mathbb A_{\kappa}^{d_1},R\psi_{\widehat{f}}F|_{\mathbb A_{\kappa}^{d_1}})\stackrel{\text{qis}}{\to} R\Gamma_{X_{\gamma}}(A_{\gamma},F_{\gamma}^{\circ})$, $F_{\gamma}^{\circ}$ the pullback of $F\in \mathbf S(A)$ via $A_{\gamma}\cong A$.\\
\indent (ii) $(R\psi_{\widehat{f}_{\mathfrak Z }}G)_0\stackrel{\text{qis}}{\to} R\Gamma(D,G|_D)$, for $G\in\mathbf S(\mathfrak Z_{\overline{\eta}})$.
\end{corollary}

\begin{proof}
By the description of $A$ and $X$, there are isomorphisms of analytic spaces $A_{\gamma}\cong A$ and $X_{\gamma}\cong X$ for a fixed $\gamma$ in $|(K^s)^*|\cap (0,1)$. These together with (\ref{eq3}), (\ref{eq4}) and Lemma \ref{lem3.1} imply (i). Also, (ii) follows from Corollary \ref{cormain}.
\end{proof}

\begin{corollary}\label{Bercor}
Keeping the assumption of Theorem \ref{maintheorem} and fixing a $\gamma\in |(K^s)^*|\cap (0,1)$, we have

\indent (i) $R\Gamma_c(\mathbb A_{\kappa}^{d_1},R\psi_{\widehat{f}}\mathbb Q_{\ell}|_{\mathbb A_{\kappa}^{d_1}})\stackrel{\text{qis}}{\to} R\Gamma_{X_{\gamma}}(A_{\gamma},\mathbb Q_{\ell})$,\\
\indent (ii) $(R\psi_{\widehat{f}_{\mathfrak Z}}\mathbb Q_{\ell})_0\stackrel{\text{qis}}{\to} R\Gamma(D,\mathbb Q_{\ell})$.
\end{corollary}

\section{Proof of Theorem \ref{maintheorem}}\label{sec5}
\subsection{Using comparison theorem} 
By Corollary \ref{Bercor}, there is a quasi-isomorphism of complexes
\begin{align}\label{eq10}
R\Gamma_c(\mathbb A_{\kappa}^{d_1},R\psi_{\widehat{f}}\mathbb Q_{\ell}|_{\mathbb A_{\kappa}^{d_1}})\stackrel{\text{qis}}{\to} R\Gamma_{X_{\gamma}}(A_{\gamma},\mathbb Q_{\ell}),
\end{align}
where $\gamma$ is fixed in $|(K^s)^*|\cap (0,1)$, $A_{\gamma}$ is the analytic subspace of $\mathbb A_{\widehat{K^s},\text{Ber}}^d$ given by $|x|\leq \gamma^{-1}$, $|y|<\gamma$, $|z|<1$ and $f(x,y,z)=t$, and $X_{\gamma}$ is defined as $A_{\gamma}$ but with $|x|< \gamma^{-1}$ in stead of $|x|\leq \gamma^{-1}$. The space $A_{\gamma}$ is a paracompact $\widehat{K^s}$-analytic space which is a union of the following increasing sequence of compact domains
$$A_{\gamma,\epsilon}:=\{(x,y,z)\in\mathbb A_{\widehat{K^s},\text{Ber}}^d: |x|\leq \gamma^{-1}, |y|\leq \gamma\epsilon, |z|\leq \epsilon, f(x,y,z)=t\},$$
for $\epsilon\in |(K^s)^*|\cap (0,1)$. The space $X_{\gamma}$ is covered by the corresponding increasing sequence 
$$X_{\gamma,\epsilon}=\{(x,y,z)\in\mathbb A_{\widehat{K^s},\text{Ber}}^d: |x|< \gamma^{-1}, |y|\leq \gamma\epsilon, |z|\leq \epsilon, f(x,y,z)=t\}.$$
Denote $B_{\gamma}:=A_{\gamma}\setminus X_{\gamma}$ and $B_{\gamma,\epsilon}:=A_{\gamma,\epsilon}\setminus X_{\gamma,\epsilon}$.

Let us consider $f^{\gamma}:=\widehat{f}_{\overline{\eta}}:A_{\gamma}\cong A\to \mathcal M(\widehat{K^s})$ and $f^{\gamma,\epsilon}$, the restriction of $f^{\gamma}$ to $A_{\gamma,\epsilon}$.

\begin{lemma}\label{lem4.1}
For any $m\geq 1$ and $F\in\mathbf S(A_{\gamma})$, there is a canonical isomorphism of groups
$$H^m_{X_{\gamma}}(A_{\gamma},F)\cong \projlim_{\epsilon\to 1}H^m_{X_{\gamma,\epsilon}}(A_{\gamma,\epsilon},F).$$
\end{lemma}

\begin{proof}
The functors $H^m_{X_{\gamma}}(A_{\gamma},-)$ are the derived functors of the global section functor $H^0_{X_{\gamma}}(A_{\gamma},-)$ defined by 
$$H^0_{X_{\gamma}}(A_{\gamma},F)=\ker(F(A_{\gamma})\to F(B_{\gamma})),$$
the kernel of the restriction homomorphism $F(A_{\gamma})\to F(B_{\gamma})$. Note that if $J$ is an injective abelian sheaf then the pullback of $J$ on $B_{\gamma}$ is acyclic and the homomorphism $J(A_{\gamma})\to J(B_{\gamma})$ is surjective. Take an injective resolution of $F$, namely $0\to F\to J^0\to J^1\to\cdots$, and consider the following commutative diagram
$$
\begin{CD}
0 \longrightarrow \ker(\alpha_0)@>d^0>> \ker(\alpha_1)@>d^1>> \ker(\alpha_2)&@>d^2>>\cdots\\
 @VVV @VVV @VVV\\
0 \longrightarrow J^0(A_{\gamma})@>d_{A_{\gamma}}^0>> J^1(A_{\gamma})@>d_{A_{\gamma}}^1>> J^2(A_{\gamma})&@>d_{A_{\gamma}}^2>>\cdots\\ 
 @VV\alpha_0V @VV\alpha_1V @VV\alpha_2V\\
0 \longrightarrow J^0(B_{\gamma})@>d_{B_{\gamma}}^0>> J^1(B_{\gamma})@>d_{B_{\gamma}}^1>> J^2(B_{\gamma})&@>d_{B_{\gamma}}^2>>\cdots
\end{CD}
$$
Then we have 
$$H^m_{X_{\gamma}}(A_{\gamma},F)=\ker\big(H^m(A_{\gamma},F)\to H^m(B_{\gamma},F)\big)\cong \ker(d^m)/\text{im}(d^{m-1}).$$

Analogously, we consider the surjections, say, $\alpha_{m,\epsilon}: J^m(A_{\gamma,\epsilon})\to J^m(B_{\gamma,\epsilon})$. There is a commutative diagram as follows, in which every vertical arrow is surjective,
$$
\begin{CD}
0 \longrightarrow \ker(\alpha_0)@>d^0>> \ker(\alpha_1)@>d^1>> \ker(\alpha_2)&@>d^2>>\cdots\\
 @VVV @VVV @VVV\\
0 \longrightarrow \ker(\alpha_{0,\epsilon})@>d_{\epsilon}^0>> \ker(\alpha_{1,\epsilon})@>d_{\epsilon}^1>> \ker(\alpha_{2,\epsilon})&@>d_{\epsilon}^2>>\cdots 
\end{CD}
$$
Here $H^m_{X_{\gamma,\epsilon}}(A_{\gamma,\epsilon},F)\cong \ker(d_{\epsilon}^m)/\text{im}(d_{\epsilon}^{m-1})$. Then we can use the arguments of \cite[Lemma 6.3.12]{Ber1} to complete the proof. Note that in this situation the following condition is satisfied: For any $0<\epsilon<1$, for any $\epsilon<\epsilon',\epsilon^{\prime\prime}<1$, the image of $H^{m-1}_{X_{\gamma,\epsilon'}}(A_{\gamma,\epsilon'},F)$ and that of $H^{m-1}_{X_{\gamma,\epsilon^{\prime\prime}}}(A_{\gamma,\epsilon^{\prime\prime}},F)$ coincide in $H^{m-1}_{X_{\gamma,\epsilon}}(A_{\gamma,\epsilon},F)$ under the restriction homomorphisms (see \cite[Lemma 7.4]{Ber3} for a similar argument).
\end{proof}

Here is an important corollary of (\ref{eq10}) and Lemma \ref{lem4.1}.

\begin{corollary}\label{cor4.2}
There is a canonical quasi-isomorphism of complexes
$$R\Gamma_c(\mathbb A_{\kappa}^{d_1},R\psi_{\widehat{f}}\mathbb Q_{\ell}|_{\mathbb A_{\kappa}^{d_1}})\stackrel{\text{qis}}{\to}\projlim_{\epsilon\to 1}R\Gamma_{X_{\gamma,\epsilon}}(A_{\gamma,\epsilon},\mathbb Q_{\ell}).$$
\end{corollary}

\begin{proof}
We deduce from (\ref{eq10}) and properties of the mapping cone functor that
\begin{align*}
R\Gamma_c(\mathbb A_{\kappa}^{d_1},R\psi_{\widehat{f}}\mathbb Q_{\ell}|_{\mathbb A_{\kappa}^{d_1}})&\stackrel{\text{qis}}{\to} R\Gamma_{X_{\gamma}}(A_{\gamma},\mathbb Q_{\ell})\\
&\cong Rf^{\gamma}_*\text{Cone}\Big(\mathbb Q_{\ell}\to i_{B_{\gamma},A_{\gamma}*}i_{B_{\gamma},A_{\gamma}}^*\mathbb Q_{\ell}\Big)\\
&\cong \text{Cone}\Big(Rf^{\gamma}_*\mathbb Q_{\ell}\to R(f^{\gamma}|_{B_{\gamma}})_*\mathbb Q_{\ell}\Big).
\end{align*}
By the universality of the projective limit, there are canonical morphisms
\begin{align*}
Rf^{\gamma}_*\mathbb Q_{\ell}&\to \projlim_{\epsilon\to 1}Rf^{\gamma,\epsilon}_*\mathbb Q_{\ell},\\
R(f^{\gamma}|_{B_{\gamma}})_*\mathbb Q_{\ell}&\to \projlim_{\epsilon\to 1}R(f^{\gamma,\epsilon}|_{B_{\gamma,\epsilon}})_*\mathbb Q_{\ell}.
\end{align*}
Here, the latter is induced from the former by restriction. Thus there is a canonical morphism of complexes
\begin{align*}
R\Gamma_c(\mathbb A_{\kappa}^{d_1},R\psi_{\widehat{f}}\mathbb Q_{\ell}|_{\mathbb A_{\kappa}^{d_1}})&\to \text{Cone}\Big(\projlim_{\epsilon\to 1}Rf^{\gamma,\epsilon}_*\mathbb Q_{\ell}\to \projlim_{\epsilon\to 1}R(f^{\gamma,\epsilon}|_{B_{\gamma,\epsilon}})_*\mathbb Q_{\ell}\Big)\\
&\cong \projlim_{\epsilon\to 1}\text{Cone}\Big(Rf^{\gamma,\epsilon}_*\mathbb Q_{\ell}\to R(f^{\gamma,\epsilon}|_{B_{\gamma,\epsilon}})_*\mathbb Q_{\ell}\Big)\\
&\cong \projlim_{\epsilon\to 1}R\Gamma_{X_{\gamma,\epsilon}}(A_{\gamma,\epsilon},\mathbb Q_{\ell}).
\end{align*}
This morphism of complexes in fact induces the cohomological isomorphisms in Lemma \ref{lem4.1}.
\end{proof}

The second part of Corollary \ref{Bercor} asserts that 
\begin{align}\label{eq11}
(R\psi_{\widehat{f}_{\mathfrak Z}}\mathbb Q_{\ell})_0\stackrel{\text{qis}}{\to} R\Gamma(D,\mathbb Q_{\ell}).
\end{align}
The space $D$ is open and locally compact, which is covered by the compact domains $D_{\epsilon}=\{z\in\mathbb A_{\widehat{K^s},\text{Ber}}^r: |z|\leq \epsilon, f(0,0,z)=t\}$, for $\epsilon\in |(K^s)^*|\cap (0,1)$. By \cite[Lem. 6.3.12]{Ber1}, there is a canonical isomorphism of cohomology groups
$$H^m(D,\mathbb Q_{\ell})\cong \projlim_{\epsilon\to 1} H^m(D_{\epsilon},\mathbb Q_{\ell})$$
for any $m\geq 0$. Thus by the same arguments as in the proof of Corollary \ref{cor4.2}, one deduces from (\ref{eq11}) that
\begin{align}\label{eq70}
(R\psi_{\widehat{f}_{\mathfrak Z}}\mathbb Q_{\ell})_0\stackrel{\text{qis}}{\to} \projlim_{\epsilon\to 1} R\Gamma(D_{\epsilon},\mathbb Q_{\ell}).
\end{align}
(Compare this with \cite[Lem. 7.4]{Ber3}.)


\subsection{Using K\"{u}nneth isomorphism} 
We now use the K\"{u}nneth isomorphism for cohomology with compact support mentioned in Proposition \ref{prop2.2}, (iii). To begin, we write $A_{\gamma,\epsilon}$ as a disjoint union $A_{\gamma,\epsilon}=A^0_{\gamma,\epsilon}\sqcup A^1_{\gamma,\epsilon}$ of analytic spaces 
\begin{align*}
A^0_{\gamma,\epsilon}&:=\{(x,y,z)\in A_{\gamma,\epsilon}: |x||y|=0\},\\
A^1_{\gamma,\epsilon}&:=\{(x,y,z)\in A_{\gamma,\epsilon}: |x||y|\not=0\}.
\end{align*}
Similarly, one can write $X_{\gamma,\epsilon}$ as a disjoint union of analytic spaces 
\begin{align*}
X^0_{\gamma,\epsilon}&:=\{(x,y,z)\in X_{\gamma,\epsilon}: |x||y|=0\},\\
X^1_{\gamma,\epsilon}&:=\{(x,y,z)\in X_{\gamma,\epsilon}: |x||y|\not=0\}.
\end{align*} 

Observe that we can write $X^0_{\gamma,\epsilon}$ as the product $Y^0_{\gamma,\epsilon}\times D_{\epsilon}$ with $D_{\epsilon}$ as in Subsection \ref{AXD} and $Y^0_{\gamma,\epsilon}:=\{(x,y)\in \mathbb A_{\widehat{K^s},\text{Ber}}^{d_1+d_2}: |x||y|=0, |x|< \gamma^{-1}, |y|\leq \gamma\epsilon\}$. By the compactness of $A^0_{\gamma,\epsilon}$, $D_{\epsilon}$, and by the K\"{u}nneth isomorphism, we have
\begin{align}\label{eq12}
R\Gamma_{X^0_{\gamma,\epsilon}}(A^0_{\gamma,\epsilon},\mathbb Q_{\ell}) \cong R\Gamma_c(X^0_{\gamma,\epsilon},\mathbb Q_{\ell})&\stackrel{\text{qis}}{\to} R\Gamma_c(Y^0_{\gamma,\epsilon}, \mathbb Q_{\ell})\otimes R\Gamma_c(D_{\epsilon},\mathbb Q_{\ell})\notag\\
&\stackrel{\text{qis}}{\to} R\Gamma_c(Y^0_{\gamma,\epsilon}, \mathbb Q_{\ell})\otimes R\Gamma(D_{\epsilon},\mathbb Q_{\ell}).
\end{align}

Decompose $Y^0_{\gamma,\epsilon}$ into a disjoint union of $Y^{0,1}_{\gamma,\epsilon}:=\{(x,0)\in \mathbb A_{\widehat{K^s},\text{Ber}}^{d_1+d_2}: |x|< \gamma^{-1}\}$ and $Y^{0,2}_{\gamma,\epsilon}:=\{(0,y)\in \mathbb A_{\widehat{K^s},\text{Ber}}^{d_1+d_2}: 0<|y|\leq \gamma\epsilon\}$.

\begin{lemma}\label{lem5.1}
(i) $R\Gamma_c(\mathbb A_{\kappa}^{d_1}, \mathbb Q_{\ell}) \stackrel{\text{qis}}{\to} R\Gamma_c(Y^{0,1}_{\gamma,\epsilon} ,\mathbb Q_{\ell})$; \qquad (ii) $R\Gamma_c(Y^{0,2}_{\gamma,\epsilon} ,\mathbb Q_{\ell})\stackrel{\text{qis}}{\to} 0$; 

\hspace*{2.28cm} (iii) $R\Gamma_c(\mathbb A_{\kappa}^{d_1}, \mathbb Q_{\ell})\stackrel{\text{qis}}{\to} R\Gamma_c(Y^0_{\gamma,\epsilon}, \mathbb Q_{\ell})$.
\end{lemma}
 
\begin{proof}
(i) For notational simplicity, let $F$ denote both constant sheaves $\mathbb Z/\ell^n\mathbb Z$ on $\mathbb A_{K^s}^{d_1}$ and on $\mathbb A_{K^s}^{d_1,\text{an}}=\mathbb A_{\widehat{K^s},\text{Ber}}^{d_1}$. The comparison theorem for cohomology with compact support \cite[Thm. 7.1.1]{Ber1} gives an isomorphism of groups
\begin{align}\label{ss}
H_c^m(\mathbb A_{K^s}^{d_1}, F)\cong H_c^m(\mathbb A_{K^s}^{d_1,\text{an}},F),
\end{align}
for any $m\geq 0$. Let $V=\mathbb A_{K^s}^{d_1,\text{an}}\setminus Y^{0,1}_{\gamma,\epsilon}$. By Proposition 5.2.6 (ii) of \cite{Ber1} (notice that Proposition \ref{prop2.2} (ii) is the $\ell$-adic version of this result), we have an exact sequence 
\begin{align}\label{sequence}
\cdots\to H_c^m(V,F)\to H_c^{m+1}(Y^{0,1}_{\gamma,\epsilon},F)\to H_c^{m+1}(\mathbb A_{K^s}^{d_1,\text{an}},F)\to H_c^{m+1}(V,F)\to\cdots.
\end{align}
We shall prove that $H_c^m(V,F)=0$ for every $m$. 

Let us choose an {\it open} covering $\{\mathscr V_i\}_{i\in \mathbb N}$ of $V=\mathbb A_{K^s}^{d_1,\text{an}}\setminus Y^{0,1}_{\gamma,\epsilon}$ defined as follows:
$$\mathscr V_i:=\{x\in \mathbb A_{K^s}^{d_1,\text{an}}: \gamma^{-1}\leq |x|<\gamma_i\},$$
where $\gamma^{-1}<\gamma_i<\gamma_j$ for every $i<j$. Choose an analogous {\it open} covering $\{\mathscr V_{ijl}\}_{l\in \mathbb N}$ of $\mathscr V_i\cap \mathscr V_j$ for each pair $i,j$. Let $\alpha_i$ and $\alpha_{ijl}$ be the open embeddings $\mathscr V_i\to V$ and $\mathscr V_{ijl}\to V$, respectively. Then the following exact sequence 
$$\bigoplus_{i,j,l}\alpha_{ijl!}(F_{\mathscr V_{ijl}})\to \bigoplus_i\alpha_{i!}(F_{\mathscr V_i})\to F_V\to 0$$
induces a exact sequence
$$
\bigoplus_{i,j,l}H_c^m(\mathscr V_{ijl},F)\to \bigoplus_iH_c^m(\mathscr V_i,F)\to H_c^m(V,F)\to 0.
$$
The \'etale cohomology groups with compact support $H_c^m(\mathscr V_{ijl},F)$ and $H_c^m(\mathscr V_i,F)$ clearly vanish for $m\geq 0$, 
thus $H_c^m(V,F)=0$ for $m\geq 0$. By (\ref{sequence}), one has $H_c^m(\mathbb A_{K^s}^{d_1,\text{an}},F)\cong H_c^m(Y^{0,1}_{\gamma,\epsilon},F)$ for $m\geq 0$, which together with (\ref{ss}) implies that $H_c^m(\mathbb A_{K^s}^{d_1}, F)\cong H_c^m(Y^{0,1}_{\gamma,\epsilon},F)$ for $m\geq 0$. Now, since $\kappa$ is algebraically closed and $K^s$ is separably closed (for fields of characteristic zero the concepts ``algebraically closed'' and ``separably closed'' coincide), applying a result of SGA$4\frac 1 2$ \cite[Cor. 3.3]{SGA}, for $m\geq 0$, $H_c^m(\mathbb A_{\kappa}^{d_1},F)\cong H_c^m(\mathbb A_{K^s}^{d_1},F)$. Therefore 
$$H_c^m(\mathbb A_{\kappa}^{d_1},F)\cong H_c^m(Y^{0,1}_{\gamma,\epsilon},F), \ m\geq 0,$$
hence the $\ell$-adic version, namely, $H_c^m(\mathbb A_{\kappa}^{d_1},\mathbb Q_{\ell})\cong H_c^m(Y^{0,1}_{\gamma,\epsilon},\mathbb Q_{\ell})$ for $m\geq 0$.

\medskip
(ii) Let us denote by $F$ the constant sheaf $\mathbb Z/\ell^n\mathbb Z$, and consider the closed immersion $\mathcal M(\widehat{K^s})\to \mathcal M(\widehat{K^s}\{\gamma^{-1}y\})$ of $\widehat{K^s}$-analytic spaces. By \cite[Cor. 4.3.2]{Ber1}, there is an isomorphism of groups 
$$H^m(\mathcal M(\widehat{K^s}),F)\cong H^m(\mathcal M(\widehat{K^s}\{\gamma^{-1}y\}),F)$$
for each $m\geq 0$. This leads an isomorphism of groups in the $\ell$-adic cohomology. Thus using the exact sequence in Proposition \ref{prop2.2} (ii), we have $H_c(Y^{0,2}_{\gamma,\epsilon},\mathbb Q_{\ell})=0$. 

\medskip
(iii) follows from (i) and (ii).
\end{proof}

\subsection{The final step of the proof}
The aim of this subsection is to prove the following  
\begin{align}\label{eq80}
R\Gamma_{X_{\gamma,\epsilon}}(A_{\gamma,\epsilon},\mathbb Q_{\ell})\stackrel{\text{qis}}{\to}  R\Gamma_{X^0_{\gamma,\epsilon}}(A^0_{\gamma,\epsilon},\mathbb Q_{\ell}).
\end{align}
Assume the quasi-isomorphism (\ref{eq80}). Then there are quasi-isomorphisms of complexes, due to Corollary \ref{cor4.2}, (\ref{eq80}), (\ref{eq12}) and Lemma \ref{lem5.1},
\begin{align*}
R\Gamma_c(\mathbb A_{\kappa}^{d_1},R\psi_{\widehat{f}}\mathbb{Q}_{\ell}|_{\mathbb A_{\kappa}^{d_1}})&\stackrel{\text{qis}}{\to} \projlim_{\epsilon\to 1} \Big(R\Gamma_c(\mathbb A_{\kappa}^{d_1}, \mathbb Q_{\ell}){\otimes} R\Gamma(D_{\epsilon},\mathbb Q_{\ell})\Big)\\
&\stackrel{\text{qis}}{\to} R\Gamma_c(\mathbb A_{\kappa}^{d_1}, \mathbb Q_{\ell}){\otimes} \projlim_{\epsilon\to 1} R\Gamma(D_{\epsilon},\mathbb Q_{\ell}).
\end{align*}
This together with (\ref{eq70}) implies Theorem \ref{maintheorem}.

\medskip
To process a proof for (\ref{eq80}), we write $R\Gamma_{X_{\gamma,\epsilon}}(A_{\gamma,\epsilon},\mathbb Q_{\ell})$ and $R\Gamma_{X^0_{\gamma,\epsilon}}(A^0_{\gamma,\epsilon},\mathbb Q_{\ell})$ in the following form: 
\begin{align*}
R\Gamma_{X_{\gamma,\epsilon}}(A_{\gamma,\epsilon},\mathbb Q_{\ell})&\stackrel{\text{qis}}{\to} Rf^{\gamma,\epsilon}_*\text{Cone} (\mathbb{Q}_{\ell,A_{\gamma,\epsilon}}\to i_{B_{\gamma,\epsilon},A_{\gamma,\epsilon} *}\mathbb{Q}_{\ell,B_{\gamma,\epsilon}}),\\
R\Gamma_{X^0_{\gamma,\epsilon}}(A^0_{\gamma,\epsilon},\mathbb Q_{\ell})&\stackrel{\text{qis}}{\to} R(f^{\gamma,\epsilon}|_{A^0_{\gamma,\epsilon}})_*\text{Cone} (\mathbb{Q}_{\ell,A^0_{\gamma,\epsilon}}\to i_{B^0_{\gamma,\epsilon},A^0_{\gamma,\epsilon} *}\mathbb{Q}_{\ell,B^0_{\gamma,\epsilon}}),
\end{align*}
where $A^0_{\gamma,\epsilon}:=\{(x,y,z)\in A_{\gamma,\epsilon}: |x||y|=0\}$ and $B^0_{\gamma,\epsilon}:=B_{\gamma,\epsilon}\cap A^0_{\gamma,\epsilon}$. To abuse notation we shall use from now on $\mathbb{Q}_{\ell}$ in stead of $\mathbb{Q}_{\ell,A_{\gamma,\epsilon}}$, $\mathbb{Q}_{\ell,B_{\gamma,\epsilon}}$, $\mathbb{Q}_{\ell,A^0_{\gamma,\epsilon}}$ or $\mathbb{Q}_{\ell,B^0_{\gamma,\epsilon}}$.  

\begin{theorem}\label{finaltheorem}
With the previous notation and hypotheses, there is a canonical quasi-isomorphism of complexes
$$Rf^{\gamma,\epsilon}_*\text{\rm Cone} (\mathbb{Q}_{\ell}\to i_{B_{\gamma,\epsilon},A_{\gamma,\epsilon} *}\mathbb{Q}_{\ell})\stackrel{\text{qis}}{\to} R(f^{\gamma,\epsilon}|_{A^0_{\gamma,\epsilon}})_*\text{\rm Cone} (\mathbb{Q}_{\ell}\to i_{B^0_{\gamma,\epsilon},A^0_{\gamma,\epsilon} *}\mathbb{Q}_{\ell}).$$
\end{theorem}

\begin{proof}
The space $A^1_{\gamma,\epsilon}:=\{(x,y,z)\in A_{\gamma,\epsilon}: |x||y|\not=0\}$ together with $A^0_{\gamma,\epsilon}$ composing a disjoint union of $A_{\gamma,\epsilon}$, there exists a canonical exact triangle  
\begin{align}\label{cet}
\to R\overline{f}^{\gamma,\epsilon}_!\text{Cone} (\mathbb{Q}_{\ell}\to i_{B^1_{\gamma,\epsilon},A^1_{\gamma,\epsilon} *}\mathbb{Q}_{\ell})\to Rf^{\gamma,\epsilon}_*\text{Cone} (\mathbb{Q}_{\ell}\to i_{B_{\gamma,\epsilon},A_{\gamma,\epsilon} *}\mathbb{Q}_{\ell})\notag\\
\to \quad R(f^{\gamma,\epsilon}|_{A^0_{\gamma,\epsilon}})_*\text{Cone} (\mathbb{Q}_{\ell}\to i_{B^0_{\gamma,\epsilon},A^0_{\gamma,\epsilon} *}\mathbb{Q}_{\ell})\to,
\end{align}
where $\overline{f}^{\gamma,\epsilon}:=f^{\gamma,\epsilon}|_{A^1_{\gamma,\epsilon}}$ and $B^1_{\gamma,\epsilon}:=B_{\gamma,\epsilon}\cap A^1_{\gamma,\epsilon}$. 
We are going to verify the following 
\begin{align}\label{tozero}
R\overline{f}^{\gamma,\epsilon}_!\text{Cone} (\mathbb{Q}_{\ell}\to i_{B^1_{\gamma,\epsilon},A^1_{\gamma,\epsilon} *}\mathbb{Q}_{\ell})\stackrel{\text{qis}}{\to} 0.
\end{align}

Let us consider the action of $\mathbb G_{m,\widehat{K^s}}^{\text{an} }$ on $\mathbb A_{\widehat{K^s},\text{Ber}}^d$ given by $\tau\cdot (x,y,z)=(\tau x,\tau^{-1}y,z)$ for $\tau\in\mathbb G_{m,\widehat{K^s}}^{\text{an} }$ and $(x,y,z)\in\mathbb A_{\widehat{K^s},\text{Ber}}^d$. This $\mathbb G_{m,\widehat{K^s}}^{\text{an} }$-action is free, since $\tau\cdot (x,y,z)=(x,y,z)$ if and only if $\tau=1$. Each orbit of the action on $A^1_{\gamma,\epsilon}$ has the following form
$$\mathbb G_{m,\widehat{K^s}}^{\text{an} }\cdot(x,y,z)\cap A^1_{\gamma,\epsilon}=\{(\tau x,\tau^{-1}y,z): \gamma^{-1}\epsilon^{-1}|y|\leq |\tau|\leq \gamma^{-1}|x|^{-1}\}$$
for $(x,y,z)\in A^1_{\gamma,\epsilon}$. Also, an orbit of $\mathbb G_{m,\widehat{K^s}}^{\text{an} }$-action on $B^1_{\gamma,\epsilon}$ is of the form 
$$\mathbb G_{m,\widehat{K^s}}^{an}\cdot(x,y,z)\cap B^1_{\gamma,\epsilon}=\{(\tau x,\tau^{-1}y,z): |\tau|=\gamma^{-1}|x|^{-1}\}$$
for $(x,y,z)\in B^1_{\gamma,\epsilon}$. Furthermore, the $\mathbb G_{m,\widehat{K^s}}^{\text{an} }$-action has the following 
\begin{property} 
Every orbit on $\mathbb A_{\widehat{K^s},\text{Ber}}^d$ intersects with $A^1_{\gamma,\epsilon}$ in a closed annulus $C$ and with $B^1_{\gamma,\epsilon}$ in a thin annulus contained in $C$.
\end{property}

Let $\mathcal P$ be the space of orbits of $\mathbb G_{m,\widehat{K^s}}^{\text{an} }$-action on $\mathbb A_{\widehat{K^s},\text{Ber}}^d$.
By Lemma \ref{fl}, $\mathcal P$ admits an obvious structure of a $\widehat{K^s}$-analytic space. The property (*) deduces that the restriction maps of the natural projection onto $\mathcal P$ on $A^1_{\gamma,\epsilon}$ and on $B^1_{\gamma,\epsilon}$, say, $a: A^1_{\gamma,\epsilon}\to \mathcal P$ and $b: B^1_{\gamma,\epsilon}\to \mathcal P$, are surjective. We remark that $\overline{f}^{\gamma,\epsilon}$ and $\overline{f}^{\gamma,\epsilon}|_{B^1_{\gamma,\epsilon}}$ factor through $a$ and $b$, respectively. Since one has a spectral sequence (the Leray spectral sequence, see Berkovich \cite[Thm. 5.2.2]{Ber1})
$$H_c^n(\mathcal P,R^ma_!\text{Cone} (\mathbb{Q}_{\ell}\to i_{B^1_{\gamma,\epsilon},A^1_{\gamma,\epsilon} *}\mathbb{Q}_{\ell})) \Rightarrow R^{n+m}\overline{f}^{\gamma,\epsilon}_!\text{Cone} (\mathbb{Q}_{\ell}\to i_{B^1_{\gamma,\epsilon},A^1_{\gamma,\epsilon} *}\mathbb{Q}_{\ell}),$$
it suffices to verify that $Ra_!\text{Cone} (\mathbb{Q}_{\ell}\to i_{B^1_{\gamma,\epsilon},A^1_{\gamma,\epsilon} *}\mathbb{Q}_{\ell})$ is quasi-isomorphic to $0$. Let us consider the following exact triangle of complexes on $\mathcal P$:
$$\to Ra_!\mathbb{Q}_{\ell}\to Rb_!\mathbb{Q}_{\ell}\to Ra_!\text{Cone} (\mathbb{Q}_{\ell}\to i_{B^1_{\gamma,\epsilon},A^1_{\gamma,\epsilon} *}\mathbb{Q}_{\ell})[+1]\to.$$
Applying the Berkovich's weak base change theorem \cite[Thm. 5.3.1]{Ber1}, we have
\begin{align*}
(R^ma_!\mathbb{Q}_{\ell})_{\lambda}\cong H_c^m(a^{-1}(\lambda),\mathbb{Q}_{\ell}),\quad (R^mb_!\mathbb{Q}_{\ell})_{\lambda}\cong H_c^m(b^{-1}(\lambda),\mathbb{Q}_{\ell})
\end{align*}
for $\lambda\in\mathcal P$ and $m\geq 0$. The embedding of the thin annulus $b^{-1}(\lambda)$ into the closed annulus $a^{-1}(\lambda)$ inducing an isomorphism on \'etale cohomology (here since $a^{-1}(\lambda)$ and $b^{-1}(\lambda)$ are compact, their \'etale cohomology and \'etale cohomology with compact support are the same), we obtain $(R^ma_!\mathbb{Q}_{\ell})_{\lambda}\cong (R^mb_!\mathbb{Q}_{\ell})_{\lambda}$. In other words, for $\lambda\in\mathcal P$ and $m\geq 0$,
$$R^ma_!\text{Cone} (\mathbb{Q}_{\ell}\to i_{B^1_{\gamma,\epsilon},A^1_{\gamma,\epsilon} *}\mathbb{Q}_{\ell})_{\lambda}\cong 0.$$
This prove (\ref{tozero}), which together with (\ref{cet}) implies the theorem.
\end{proof}

\begin{lemma}\label{fl}
There is a natural structure of an analytic space on the quotient 
$$\mathcal P=(\mathbb A_{\widehat{K^s},\text{Ber}}^{d_1+d_2}\setminus\{0\})\times \mathbb A_{\widehat{K^s},\text{Ber}}^{d_3}/\mathbb G_{m,K^s}^{\text{an}}.$$
\end{lemma}
\begin{proof}
We endow $\mathcal P$ with the quotient topology, then obviously it is a compact Hausdorff space. The construction of analytic structure on $\mathcal P$ is analogous to that of the projective analytic spaces $\mathbb P_{\widehat{K^s},\text{Ber}}^d$, where the natural $\mathbb G_{m,\widehat{K^s}}^{\text{an} }$-action on $\mathbb A_{\widehat{K^s},\text{Ber}}^d$ is replaced by the $\mathbb G_{m,\widehat{K^s}}^{\text{an} }$-action given by $\tau\cdot(x,y,z)=(\tau x, \tau^{-1}y,z)$, which is also free. See \cite{Tem} for the construction in detail of $\mathbb P_{K,\text{Ber}}^d$.
\end{proof}

\end{document}